\documentclass{amsart}
\usepackage{graphicx}
\usepackage{amssymb}

\newtheorem{thm}{Theorem}[section]
\newtheorem{cor}[thm]{Corollary}
\newtheorem{lem}[thm]{Lemma}
\newtheorem{prop}[thm]{Proposition}
\theoremstyle{definition}

\theoremstyle{remark}
\newtheorem{rem}[thm]{Remark}

\begin{document}

\title[Regularity results for $u_{11}u_{22} = 1$]{Regularity results for the equation $u_{11}u_{22} = 1$}
\author{Connor Mooney}
\author{Ovidiu Savin}
\address{Department of Mathematics, ETH Z\"{u}rich, Z\"{u}rich, Switzerland 8092}
\address{Department of Mathematics, Columbia University, New York, NY 10027}
\email{\tt connor.mooney@math.ethz.ch}
\email{\tt savin@math.columbia.edu}

% ----------------------------------------------------------------
\begin{abstract}
We study the equation $u_{11}u_{22} = 1$ in $\mathbb{R}^2$. Our results include an interior $C^2$ estimate, classical solvability of the Dirichlet problem,
and the existence of non-quadratic entire solutions. We also construct global singular solutions to the analogous equation in higher dimensions. 
At the end we state some open questions.
\end{abstract}
% ----------------------------------------------------------------

\maketitle

%%%%%%%%%%%%%%%%%%%%%%%%%%%%%%%%%%%%%%%%%%%%%%%%%%%%%%%%%%%%%%%%%%%%%%%%%%%%%%%%%%%
\section{Introduction}
 In this paper we study the equation
\begin{equation}\label{Equation}
(u_{11} u_{22})^{1/2} = 1
\end{equation}
together with its higher-dimensional versions. We assume $u: \Omega \subset \mathbb{R}^2 \rightarrow \mathbb{R}$ is continuous and convex when restricted to lines in the coordinate directions. On this class of functions the equation (\ref{Equation}) is
elliptic and concave, and the ellipticity constants may degenerate when $D^2u \rightarrow \infty$. Our interest in (\ref{Equation}) comes from the study of interior $C^2$ estimates for concave equations.

Equation (\ref{Equation}) shares several interesting features with the complex Monge-Amp\`{e}re equation. One is that
solutions are not convex. Another is that solutions have two different types of invariances: 
$$u(x_1,\,x_2)+ a \, x_1x_2 \quad \mbox{for any constant $a$, }$$ 
and 
$$u(\lambda x_1,\, \lambda^{-1}x_2) \quad \mbox{for any constant $\lambda \ne 0$}.$$ 
Notice that any Hessian $D^2u(x_0)$ can be mapped into $I$ after using these invariances. In some sense equation (\ref{Equation}) can be viewed as an interpolation between the Laplace equation and the real Monge-Amp\`{e}re equation. 

There are however some key differences between the equation we consider and the real Monge-Amp\`{e}re equation. Calabi's theorem states that solutions to the real Monge-Amp\`{e}re equation are very rigid: the only global solutions are quadratic polynomials. 
In contrast, there are nontrivial global solutions to (\ref{Equation}) which at infinity have subquadratic growth along the axes and superquadratic growth along the diagonals (see Theorem \ref{Entire}).
Another important difference concerns continuity estimates near $\partial \Omega$. The Dirichlet problem for (\ref{Equation}) is well posed if the intersection of $\Omega$ with any horizontal or vertical line is a single segment.
If we assume that $u=0$ on $\partial \Omega$ then solutions of (\ref{Equation}) do not have any uniform modulus of continuity near the boundary. On the other hand, for the real Monge-Amp\`{e}re equation, uniform H\"{o}lder estimates are a consequence of Alexandrov's estimate.

There are many important nonlinear concave equations for which it is not known whether a Pogorelov-type interior estimate holds
(that is, if $u=0$ on $\partial\Omega$, whether $D^2u(x)$ is bounded in terms of the distance from $x$ to $\partial \Omega$ and the diameter of $\Omega$). Equation (\ref{Equation}) can be viewed as a simplified model for such equations.
In this work we establish a pure interior $C^2$ estimate in 2D (see Theorem \ref{C2Estimate}). We plan to investigate the higher dimensional case in future work, and this could provide some insight into other similar equations.

We now state our results. The first is an interior a priori $C^2$ estimate in 2D.
\begin{thm}\label{C2Estimate}
Assume that $u \in C^4(B_1)$ solves (\ref{Equation}) in $B_1 \subset \mathbb{R}^2$. Then
$$\|u\|_{C^2(B_{1/2})} < C(\|u\|_{L^{\infty}(B_1)}).$$
\end{thm}

\noindent The proof of Theorem \ref{C2Estimate} is a maximum principle argument for a second-order quantity. We introduce a cutoff function motivated
by the partial Legendre transform, which takes the Monge-Amp\`{e}re equation to the Laplace equation in $2D$. As a corollary we obtain a Liouville theorem for solutions with quadratic growth.
\begin{cor}\label{Liouville}
Assume that $u$ is a smooth solution to (\ref{Equation}) on $\mathbb{R}^2$. If in addition $|u| < C(1 + |x|^2)$ for some constant $C > 0$, then $u$ is a quadratic polynomial.
\end{cor}

Our second result is the classical solvability of the Dirichlet problem. We say that a continuous function $w$ is coordinate-convex on $\mathbb{R}^n$ if $w$ is convex when 
restricted to lines in the coordinate directions. We say that a smooth function $w$ is uniformly coordinate-convex if $w_{ii} \geq c > 0$ on $\mathbb{R}^n$ for each
$i = 1,\,...,\,n$. Finally, we say that a domain $\Omega \subset \mathbb{R}^n$ is uniformly coordinate-convex if $\Omega$ is a connected component
of $\{w < 0\}$ for some smooth, uniformly coordinate-convex function $w$. We show:
\begin{thm}\label{DP}
Let $\Omega \subset \mathbb{R}^2$ be a bounded, uniformly coordinate-convex domain. Let $\varphi \in C^{\infty}(\mathbb{R}^2)$. Then there exists a unique coordinate-convex solution 
in $C^{\infty}(\Omega) \cap C(\overline{\Omega})$ to
$$(u_{11}u_{22})^{1/2} = 1 \text{ in } \Omega, \quad u|_{\partial \Omega} = \varphi.$$
\end{thm}

\begin{rem}
 The domain $\Omega$ need not be smooth. Consider for example a connected component of $\{|x|^4 + |x|^2 - 4x_1x_2 < 0\} \subset \mathbb{R}^2$.
\end{rem}

\noindent As a consequence of Theorem \ref{C2Estimate} and Theorem \ref{DP} we obtain local $C^{\infty}$ regularity and derivative 
estimates of all orders for viscosity solutions of (\ref{Equation}).
\begin{cor}\label{LocalReg}
 Let $u \in C(\overline{B_1})$ be a viscosity solution to (\ref{Equation}) in $B_1 \subset \mathbb{R}^2$. Then $u$ is in fact smooth, and we have
 $$\|u\|_{C^k(B_{1/2})} \leq C(k,\,\|u\|_{L^{\infty}(B_1)}).$$
\end{cor}

\begin{rem}
The Dirichlet problem for (\ref{Equation}), and its higher dimensional versions, is uniquely solvable in the class of viscosity solutions e.g. when
$\Omega$ is uniformly convex and the boundary data $\varphi$ are smooth. This follows from general theory (see \cite{I}).
\end{rem}

To obtain solutions to (\ref{Equation}) that are smooth up to $\partial\Omega$, it suffices to obtain boundary $C^2$ estimates.
Caffarelli, Nirenberg and Spruck accomplished this for a large class of Hessian equations in \cite{CNS1}, \cite{CNS2}.
We can hope that the following Caffarelli-Nirenberg-Spruck type result holds:
{\it Let $\Omega$ be a uniformly coordinate-convex, $C^3$ domain, and $\varphi \in C^3(\partial\Omega)$. Then the Dirichlet problem for (\ref{Equation}) is uniquely solvable and the solution $u$ is of class $C^{2,\alpha}(\overline{\Omega})$.}
However, the boundary $C^2$ estimate for (\ref{Equation}) seems to be tricky, even in two dimensions (see Remark \ref{GlobalC2}). 
To prove Theorem \ref{DP} we instead use Theorem \ref{C2Estimate} and an approximation method developed by Lions for the real Monge-Amp\`{e}re equation \cite{L}.

Our next theorem concerns global solutions. Results in the spirit of the Calabi theorem that global solutions to $\det D^2u = 1$
in $\mathbb{R}^n$ are quadratic polynomials are closely connected to regularity questions. 
Interestingly, (\ref{Equation}) admits non-quadratic entire solutions.
\begin{thm}\label{Entire}
There exist non-quadratic entire solutions to (\ref{Equation}).
\end{thm}
\noindent By Corollary \ref{Liouville}, any such solution grows super-quadratically at $\infty$. The solution we construct grows $\sim |x|^2\log |x|$ along the diagonal lines $x_1^2 = x_2^2$.

Finally, we show that the local regularity results in two dimensions are false for the analogous problem
\begin{equation}\label{HigherD}
 \left(\Pi_{i = 1}^n u_{ii}\right)^{1/n} = 1, \quad u \text{ coordinate-convex}
\end{equation}
in higher dimensions $n \geq 3$.

\begin{thm}\label{Nonclassical}
In dimensions $n \geq 3$ there exist global, non-classical viscosity solutions to (\ref{HigherD}).
\end{thm}
\noindent Our example can be viewed as an analogue of the well-known Pogorelov example for the real Monge-Amp\`{e}re equation.
A new difficulty in our case is that (\ref{HigherD}) is not rotation-invariant. Another difference is that there are no global singular solutions to the real Monge-Amp\`{e}re equation $\det D^2u = 1$.
There are global singular solutions to the complex Monge-Amp\`{e}re equation $\det \partial\overline{\partial}u = 1$ in $\mathbb{C}^n$
for all $n \geq 2$ (see e.g. \cite{B}), in contrast with the real case.

We prove each theorem in a separate section below. We delay a standard argument in the proof of Theorem \ref{DP} to the Appendix. 
In the last section we state some interesting open problems motivated by this work.

%%%%%%%%%%%%%%%%%%%%%%%%%%%%%%%%%%%%%%%%%%%%%%%%%%%%%%%%%%%%%%%%%%%%%%%%%%%%%%%%%%%
\section{Interior $C^2$ Estimate}\label{C2EstimateSection}

In this section we prove Theorem \ref{C2Estimate}. We begin by showing that solutions
to (\ref{Equation}) are strictly convex on horizontal and vertical lines.

\begin{lem}\label{StrictConvexity}
Assume that $u$ is a viscosity solution of (\ref{Equation}) in $B_1 \subset \mathbb{R}^2$. Then
\begin{equation}\label{StrictConvexityInequality}
 u(e_2/2) + u(-e_2/2) - 2u(0) \geq \delta > 0,
\end{equation}
where $\delta$ depends only on $\|u\|_{L^{\infty}(B_1)}$.
\end{lem}

\begin{rem}
 This result is special to $2D$; see the example in Section \ref{SingularSolutionSection}.
\end{rem}

\begin{proof}
Assume by way of contradiction that the lemma is false. Then there exists a uniformly bounded sequence of solutions to (\ref{Equation}) that violate
(\ref{StrictConvexityInequality}) for arbitrarily small $\delta$. By coordinate-convexity, this sequence is locally uniformly Lipschitz
and thus converges locally uniformly to a bounded viscosity solution $v$ of (\ref{Equation}) in $B_1$. After subtracting
a linear function, we may assume that $v(0,\,x_2) = 0$ for $|x_2| \leq 1/2$.

Now take
$$g_{\lambda}(x_1,\,x_2) = \lambda\, (x_1 (\log(x_1^{-1}))^{1/2})\,(4x_2^2 - 1) + \lambda^{-1} x_1,$$
with $\lambda > 0$ small. A short computation gives 
$$0 \leq (g_{\lambda})_{11}(g_{\lambda})_{22} \leq 6\lambda$$
in $R := (0,\, 1/4) \times (-1/2,\, 1/2).$ In addition, for $\lambda$ small we have $g_{\lambda} \geq v$ on $\partial R$ by the local Lipschitz
regularity of $v$ (which follows from coordinate-convexity). 
By the maximum principle, $v \leq g_{\lambda}$ in $R$ for $\lambda$ small. However, $\partial_1 g_{\lambda} \rightarrow -\infty$ near $(0,\,0)$, contradicting the coordinate-convexity of $v$.
\end{proof}

We now prove Theorem \ref{C2Estimate}. Lemma \ref{StrictConvexity} is used to justify our choice of cutoff function.

\begin{proof}[{\bf Proof of Theorem \ref{C2Estimate}:}]
It suffices to show
\begin{equation}\label{SimplifiedC2}
 u_{11}(0) < C(\|u\|_{C^1(B_1)}).
\end{equation}
Indeed, the bounds $\|u_{ii}\|_{L^{\infty}(B_{3/4})} < C(\|u\|_{L^{\infty}(B_1)})$ for $i = 1,\,2$ follow by the invariance of (\ref{Equation})
under $(x_1,\,x_2) \rightarrow (x_2,\,x_1)$ and a standard covering argument.
The equation (\ref{Equation}) then becomes uniformly elliptic in $B_{3/4}$, and the full $C^2$ estimate follows by classical theory of uniformly
elliptic PDE in $2D$, or by the concavity of the equation (see Remark \ref{HigherDerivatives} below).

We may assume after subtracting a linear function that 
\begin{equation}\label{Normalize}
u(0) = u_2(0) = 0.
\end{equation}
\noindent Let
 $$Lv = \sum_{i = 1}^2 \frac{v_{ii}}{u_{ii}}$$
denote the linearized equation. By differentiating (\ref{Equation}) once we obtain
\begin{equation}\label{DiffOnce}
 L(u_k) = 0, \quad k = 1,\,2.
\end{equation}
Differentiating (\ref{Equation}) twice in the $e_1$ direction we get
\begin{equation}\label{DiffTwice}
 L(u_{11}) = \frac{1}{u_{11}^2}u_{111}^2 + u_{11}^2 u_{122}^2.
\end{equation}

Let $\eta$ be a $C^2$ function such that $\eta(0) > 0$ and the connected component $\mathcal{U}$ of $\{\eta > 0\}$ containing the origin is compactly contained in $B_{3/4}$.
(We will choose an appropriate $\eta$ later). Let
$$M := \log u_{11} + \frac{\sigma}{2}u_1^2 + \log \eta,$$
with $\sigma > 0$. Then $M$ attains its maximum in $\mathcal{U}$ at some point $x_0$. At $x_0$ we have
\begin{equation}\label{GradientZero}
0 = M_i = \frac{u_{11i}}{u_{11}} + \sigma u_1u_{1i} + \frac{\eta_i}{\eta}, \quad i = 1,\,2
\end{equation}
and
$$0 \geq L(M) = \frac{1}{u_{11}}L(u_{11}) - \frac{u_{11i}^2}{u_{11}^2u_{ii}} + \sigma u_{11}(1 + u_{12}^2) + \frac{L(\eta)}{\eta} - \frac{\eta_i^2}{\eta^2u_{ii}}.$$
Here we used the equation and its derivative (\ref{DiffOnce}). Using the twice-differentiated equation (\ref{DiffTwice}) to simplify the first two terms we obtain
$$\sigma u_{11}(1 + u_{12}^2) - \frac{u_{112}^2}{u_{11}^2}u_{11} \leq \frac{\eta_i^2}{\eta^2 u_{ii}} - \frac{L(\eta)}{\eta}.$$
Using condition (\ref{GradientZero}) we can estimate the second term from below by $-2\sigma^2 u_1^2 u_{12}^2 u_{11} - 2\frac{\eta_2^2}{\eta^2} u_{11}$. Multiplying by $\eta^2u_{11}$ we arrive at
$$\sigma \eta^2u_{11}^2(1 + (1- 2\sigma u_1^2)u_{12}^2) \leq \eta_1^2 + 3\eta_2^2u_{11}^2 - \eta L(\eta)u_{11}.$$

We now specify $\eta$. By (\ref{Normalize}) and Lemma \ref{StrictConvexity} (appropriately rescaled), we have
$u(\pm e_2/2) \geq \delta > 0$. We claim that
$$\eta := 1 - A(x_1^2 + u_2^2)/2$$
satisfies the required conditions for some $A(\|u\|_{C^1(B_1)})$. Indeed, $u(x_1,0) < C|x_1|$ and $u(x_1, \pm 1/2) > \delta - C|x_1|,$
where $C = \|u\|_{C^1(B_1)}$. Coordinate-convexity implies that $|u_2|(x_1,\,\pm 1/2) > \delta$ for $|x_1| < \delta / 4C$. 
Thus, $A(x_1^2 + u_2^2)/2 > 1$ on the boundary of $(-\delta/4C,\, \delta / 4C) \times (-1/2,\,1/2)$ for $A$ large depending on $\|u\|_{C^1(B_1)}$.

Our choice of $\eta$ gives
$$\eta_1^2 = A^2(x_1 + u_2 u_{12})^2, \quad \eta_2^2 u_{11}^2 = A^2u_2^2.$$
Using the linearized equation (\ref{DiffOnce}) we have
$$-\eta L(\eta)u_{11} = A\,\eta (2 + u_{12}^2).$$
Putting these together we obtain
$$\sigma u_{11}^2\eta^2 (1 + (1 - 2\sigma u_1^2)u_{12}^2) \leq C(\|u\|_{C^1(B_1)})(1 + u_{12}^2).$$
By choosing $\sigma$ small depending on $\sup_{B_1} |u_1|$ we obtain
\begin{equation}\label{2DPogorelov}
\eta^2u_{11}^2(x_0) < C(\|u\|_{C^1(B_1)}).
\end{equation}
We conclude that
$$u_{11}(0) \leq \eta u_{11}e^{\sigma u_1^2/2}(0) \leq \eta u_{11}e^{\sigma u_1^2/2}(x_0) \leq C(\|u\|_{C^1(B_1)}).$$
\end{proof}

\begin{rem}
 Our choice of $\eta$ is motivated by the partial Legendre transform, which takes the Monge-Amp\`{e}re equation
 to the Laplace equation in $2D$. Roughly, to estimate $u_{11}$ from above we'd like to estimate $u_{22}$ from below. This is the same as obtaining $C^2$ estimates
 for the partial Legendre transform ``taken in the $e_2$ direction.'' Since the transformed coordinates are $(p_1,\,p_2) = (x_1,\,u_2)$,
 it is natural to seek cutoff functions depending on $x_1$ and $u_2$.  
\end{rem}

\begin{rem}\label{HigherDerivatives}
 Estimates for all the higher derivatives in terms of $\|u\|_{L^{\infty}(B_1)}$ follow from uniform ellipticity and either classical $2D$ theory
 (see e.g. \cite{GT}, Chapter $17$) or the theory of concave equations (see e.g. \cite{CC}, Chapter $6$).

 More precisely, a derivative $u_e$ of $u$ solves the uniformly elliptic equation $L(u_{e}) = 0$. Such equations enjoy interior
 $C^{1,\,\alpha}$ estimates in $2D$, giving $C^{2,\,\alpha}$ estimates for $u$. Schauder theory can be used to estimate all higher derivatives.
 Alternatively, the full $C^2$ estimates can be obtained using the concavity of the equation and the weak Harnack inequality of Krylov-Safonov
 for the second derivatives $u_{ee}$ of $u$. Then $C^{2,\alpha}$ estimates follow by the Evans-Krylov theorem, and higher regularity as before.
\end{rem}

 Corollary \ref{Liouville} follows from Theorem \ref{C2Estimate} by considering the rescalings $R^{-2}u(Rx)$ and applying Remark \ref{HigherDerivatives}.
 We will remove the assumption that $u$ is $C^4$ in the next section by solving the Dirichlet problem. 

%%%%%%%%%%%%%%%%%%%%%%%%%%%%%%%%%%%%%%%%%%%%%%%%%%%%%%%%%%%%%%%%%%%%%%%%%%%%%%%%%%%
\section{The Dirichlet Problem}\label{DPSection}

In this section we prove Theorem \ref{DP}. We use an idea of Lions based on solving global approximating problems \cite{L}.

Let $C^{\infty}_b(\mathbb{R}^2)$ be the space of smooth functions $\psi$ on $\mathbb{R}^2$ satisfying $\|\psi\|_{C^k(\mathbb{R}^2)} < \infty$ for all $k$.
Assume without loss of generality that $\varphi \in C^{\infty}_0(\mathbb{R}^2)$.
Finally, let $\rho$ be a smooth nonnegative function that vanishes on $\overline{\Omega}$, is positive on $\mathbb{R}^2 \backslash \overline{\Omega}$, and 
is $1$ outside a neighborhood of $\overline{\Omega}$. (We note that such a function exists for any bounded domain $\Omega$.) The key proposition is:

\begin{prop}\label{GlobalApproximateProblem}
 For all $\epsilon > 0$, there exists a solution in $C^{\infty}_b(\mathbb{R}^2)$ of
\begin{equation}\label{ApproxEquation}
 (u^{\epsilon}_{11} - \epsilon^{-1}(u^{\epsilon} - \varphi)\rho)(u^{\epsilon}_{22} - \epsilon^{-1}(u^{\epsilon} - \varphi)\rho) = 1
\end{equation}
with $u_{ii} - \epsilon^{-1}(u^{\epsilon} - \varphi)\rho > 0$ for $i = 1,\,2$.
\end{prop}

\noindent Heuristically, the additional terms in (\ref{ApproxEquation}) ``penalize'' the solution for deviating from $\varphi$ outside $\Omega$.

\begin{proof}[{\bf Proof of Proposition \ref{GlobalApproximateProblem}}:]
 It suffices to obtain a global $C^2$ estimate for solutions in $C^{\infty}_b(\mathbb{R}^2)$ of
 \begin{equation}\label{Model}
 \sum_{i = 1}^2 \log(u_{ii} - \epsilon^{-1}(u-\varphi)\rho) = \log(f),
 \end{equation}
 where $f \in C^{\infty}_b(\mathbb{R}^2) \cap \{f \geq 1\},$ and the estimate depends only on $\epsilon$ and $\|f\|_{C^2(\mathbb{R}^n)}$.
 Here we assume $u_{ii} > \epsilon^{-1}(u-\varphi)\rho$ for $i = 1,\,2$.
 Global estimates for the higher derivatives of $u$ then follow by classical uniformly elliptic theory (see Remark \ref{HigherDerivatives}).
 The existence in $C^{\infty}_b(\mathbb{R}^2)$ of solutions to (\ref{ApproxEquation}) follows easily by the method of continuity.
 For the details of this argument, see the Appendix (Section \ref{Appendix}).
 
 Let $w$ be a defining function of $\Omega$ (that is, $\Omega$ is a connected component of $\{w < 0\}$ and $w_{ii} \geq c > 0$ on $\mathbb{R}^2$) 
 and we let $\tilde{w} \in C^{\infty}_0(\mathbb{R}^2)$ agree with $w$ in a neighborhood of $\overline{\Omega}$.
 Below, $C$ and $K$ will denote large constants depending of $\|f\|_{C^2(\mathbb{R}^2)}$ and $\epsilon$.

 {\bf $C^0$ Estimate:} Let $A^{ii} = 1/(u_{ii} - \epsilon^{-1}(u-\varphi)\rho)$. Note that
 $$A^{ii}u_{ii} - \epsilon^{-1}\left(\sum_{i = 1}^2 A^{ii}\right)\rho\,u = 2 - \epsilon^{-1}\left(\sum_{i=1}^2 A^{ii}\right)\rho\, \varphi.$$
 It is easy to check that $C\tilde{w} - h$ is a subsolution to this equation when $h \geq K$, for
 appropriate large constants $C,\, K$. (Here $h$ is constant). For some $h$, we have that $\inf_{\mathbb{R}^2} (u - (C\tilde{w} - h)) = 0$. Assume by way of contradiction
 that $h \geq K$. Then by the maximum principle, $u > -h$ outside a large ball and approaches $-h$ at a sequence of points going to $\infty$. 
 By sliding a paraboloid with Hessian $-I$ centered at a point far from the origin where $u$ is close to $-h$ from below until it touches $u$,
 we can find a point where $D^2u > -I$ and $\rho = 1$, but $u < 1-K$. This contradicts the equation for $K$ large. We conclude that $u \geq C\tilde{w} - K$.
 
 For the estimate from above, use $K - C\tilde{w}$ as a barrier and argue in the same way.

 {\bf $C^1$ Estimate:} This follows easily from the $C^0$ estimate and the semiconcavity of $u$ in coordinate directions. (Recall that
 $u_{ii} > \epsilon^{-1}(u-\varphi)\rho$ for $i = 1,\,2$.)

 {\bf $C^2$ Estimate:} We have $u_{kk} > -C$ for $k = 1,\,2$ by the $C^0$ estimate. Differentiating (\ref{Model}) twice and using the $C^1$ estimate gives
 $$A^{ii}(u_{kk})_{ii} - \epsilon^{-1}\left(\sum_{i = 1}^2 A^{ii}\right)\rho u_{kk} \geq - C \sum_{i = 1}^2 A^{ii}.$$
 (Here we dropped a positive expression that is quadratic in third derivatives on the right side. The positivity
 is a consequence of the concavity of the equation).
 Using a barrier of the form $-C\tilde{w} + K$ and arguing as in the $C^0$ estimate gives an upper bound for $u_{kk}$.
 The equation (\ref{Model}) then becomes uniformly elliptic, so the full $C^2$ bound follows from classical theory (see Remark \ref{HigherDerivatives}).
\end{proof}

We now prove Theorem \ref{DP} by taking a limit of the solutions $u^{\epsilon}$ from Proposition \ref{GlobalApproximateProblem}.

\begin{proof}[Proof of Theorem \ref{DP}]
 The proof is a refinement of the $C^0$ estimate from Proposition \ref{GlobalApproximateProblem}. We have
 $$A^{ii}u^{\epsilon}_{ii} - \epsilon^{-1}\left(\sum_{i = 1}^2 A^{ii}\right)\rho u^{\epsilon} = 2 - \epsilon^{-1}\left(\sum_{i = 1}^2 A^{ii}\right)\rho \varphi,$$
 where $A^{ii} = 1/(u^{\epsilon}_{ii} - \epsilon^{-1}(u^{\epsilon} - \varphi)\rho)$. 
 Let $\tilde{w}_{\delta} \in C^{\infty}_0(\mathbb{R}^2)$ agree with $w$ in a small neighborhood of $\overline{\Omega}$
 and satisfy $|\tilde{w}_{\delta}| < \delta$ on $\mathbb{R}^2 \backslash \Omega$. Then $C\tilde{w}_{\delta} + \varphi - 2C\delta$ is a subsolution of the above equation
 for $\epsilon$ small depending on $\delta$. By the maximum principle, $(\varphi - u^{\epsilon})^+$ converges uniformly to zero on $\mathbb{R}^2 \backslash \Omega$ 
 as $\epsilon \rightarrow 0$. (The unboundedness of the domain is not an issue, since we are working in $C^{\infty}_b(\mathbb{R}^2)$;
 argue as in the $C^0$ estimate from the proof of Proposition \ref{GlobalApproximateProblem}).

 Similarly, $-C\tilde{w}_{\delta} + \varphi + 2C\delta$ is a super-solution of this equation for $\epsilon$ small depending on $\delta$. 
 We conclude by the maximum principle that $u^{\epsilon}$ converge uniformly to $\varphi$ on $\mathbb{R}^2 \backslash \Omega$.

 Since $u^{\epsilon}$ solve (\ref{Equation}) in $\Omega$ and converge uniformly on $\partial \Omega$, we have by the maximum principle that 
 $\{u^{\epsilon}\}$ is Cauchy in $C^0(\mathbb{R}^2)$.
 The $u^{\epsilon}$ thus converge as $\epsilon \rightarrow 0$ to a continuous function on $\mathbb{R}^2$ that agrees with $\varphi$ on $\mathbb{R}^2 \backslash \Omega$.

 Finally, by Theorem \ref{C2Estimate}, in any $\Omega' \subset \subset \Omega$, we have derivative estimates of all orders
 for $u^{\epsilon}$ that are independent of $\epsilon$.
 We conclude that the limit is smooth in $\Omega$ and solves (\ref{Equation}) classically. The uniqueness follows from the maximum principle.
\end{proof}

Corollary \ref{LocalReg} follows from Theorem \ref{C2Estimate} and Theorem \ref{DP} by approximating the boundary data with smooth functions, solving the Dirichlet problem, and
taking a limit.
 
\begin{rem}
 In higher dimensions, the above techniques show that we can approximate any viscosity solution to the equation
 $$\left(\Pi_{i = 1}^n u_{ii}\right)^{\frac{1}{n}} = 1 \text{ in } \Omega \subset \mathbb{R}^n \text{ bounded, uniformly coordinate-convex}, \quad u|_{\partial \Omega} = \varphi$$
 by smooth solutions (with different boundary data). In the case $n \geq 3$ the classical solvability remains open due to the lack of an interior $C^2$ estimate 
 (which is false without e.g. hypotheses on boundary data; see the example in Section \ref{SingularSolutionSection}).
\end{rem}

\begin{rem}\label{GlobalC2}
 An interesting question is whether $C^2$ estimates hold on $\partial \Omega$, even in the simple case that $\Omega = B_1 \subset \mathbb{R}^2$ and $\varphi$ is smooth.
 It seems that the main difficulty is to estimate the mixed second derivatives at points with a horizontal or vertical tangent line.
 
 Boundary gradient and tangential second derivative estimates are standard.
 At points with a horizontal or vertical tangent line, we can also estimate the normal second derivative. 
 To see this, assume for simplicity that $\partial \Omega = \{(x_1,\,x_1^2/2)\}$ near $0$,
 and that $u(0) = 0,\, \nabla u(0) = 0$. By the equation it suffices to bound $u_{11}(0)$ from below. By subtracting a multiple of $x_1x_2$ we may assume that the cubic part in the expansion of the boundary data vanishes. 
 If $u_{11}(0) = 0$, then $u \sim x_1^4$ along $\partial\Omega$. It follows that $\{u < h\}$ contains a box $Q$ centered on the $x_2$ axis with area $\sim h^{3/4} >> h$ for $h$ small. It is easy to construct
 a convex quadratic polynomial $P$ such that $P > h$ on $\partial Q$, $P = 0$ in the center of $Q$, and $P_{11}P_{22} << 1$. By coordinate convexity, $u \geq 0$ in the center of $Q$, so this contradicts
 the maximum principle.
 
 To bound $u_{12}(0)$ it is natural to consider a tangential derivative $u_{\tau} = u_1 + x_1u_2$. 
 However, the right hand side of the linearized equation for $u_{\tau}$ is $2u_{12}/u_{11}$, which is not
 controlled by the right hand sides for the usual quantities $u_1^2$ and $|x|^2$. 
 One can instead get estimates that degenerate near $0$ by observing that the tangential derivative $x_1u_1 + 2x_2u_2$ solves the linearized equation with constant right hand side. 
 This leads to the bounds
 $$C^{-1}x_1^2 \leq u_{11} \leq C, \quad C^{-1} \leq u_{22} \leq Cx_1^{-2}, \quad |u_{12}| \leq Cx_1^{-1}$$
 on $\partial\Omega$ near $0$. It is unclear how to get bounds that extend all the way to the origin.
\end{rem}

%%%%%%%%%%%%%%%%%%%%%%%%%%%%%%%%%%%%%%%%%%%%%%%%%%%%%%%%%%%%%%%%%%%%%%%%%%%%%%%%%%

\section{Entire Solutions}\label{BernsteinSection}

In this section construct non-quadratic entire solutions to (\ref{Equation}) in $\mathbb{R}^2$. We search for solutions of the form
$$u(x_1,\,x_2) = f(x_1)f(x_2).$$
It suffices to find a global solution to the ODE
\begin{equation}\label{2DODE}
ff'' = 1, \quad f(0) = 1, \quad f'(0) = 0.
\end{equation}
It is easy to check that the solution to (\ref{2DODE}) is given by
$$f(s) = H^{-1}(\sqrt{2}|s|),$$
where
$$H(t) = \int_1^t \frac{1}{\sqrt{\log (x)}}\,dx.$$
The function $f$ is positive, convex, even and analytic, and 
$$f \sim \sqrt{2} |s| \sqrt{\log |s|}$$ 
for $s$ large. In particular, $u \sim r^2 \log r$ on the diagonal lines $x_1^2 = x_2^2$, and $u \sim r \sqrt{\log r}$ on the coordinate axes, for $r$ large.

\begin{rem}
There are also explicit solutions to (\ref{Equation}) in the box $\Omega := [-1,\,1]^2$ with $u|_{\partial \Omega} = 0$, of the form $u = g(x_1)g(x_2)$. 
Here $g < 0$ on $(-1,\,1)$ and solves $g''g = -1,\, g(\pm 1) = 0$. A direct computation gives $g = \lambda_0^{-1}G^{-1}(\lambda_0 |x|)$ for some $\lambda_0 > 0$, 
where $G(t) = \int_{-1}^t [\log (x^{-2})]^{-1/2}\,dx,\, -1 \leq t \leq 0$. Notice that $|\nabla u| \rightarrow \infty$ on $\partial\Omega$. 
We used a variant of this solution to establish strict coordinate convexity for solutions to (\ref{Equation}) in 2D (see Lemma \ref{StrictConvexity}).
\end{rem}

%%%%%%%%%%%%%%%%%%%%%%%%%%%%%%%%%%%%%%%%%%%%%%%%%%%%%%%%%%%%%%%%%%%%%%%%%%%%%%%%%
\section{Singular Solutions in Higher Dimensions}\label{SingularSolutionSection}

In this section we construct a non-classical global (viscosity) solution in $\mathbb{R}^3$ to the equation
$$u_{11}u_{22}u_{33} = 1, \quad u \text{ coordinate-convex}.$$ 
This shows that the interior regularity results for $2D$ are false in higher dimensions.

Our example is inspired by the Pogorelov example for the real Monge-Amp\`{e}re equation \cite{P}. We search for solutions of the form
$$u(x_1,\,x_2,\,x_3) = w(x_1,\,x_2)h(x_3).$$
The problem reduces to constructing solutions to
$$ww_{11}w_{22} = 1 \text{ on } \mathbb{R}^2, \quad h^{2}h'' = 1 \text{ on } \mathbb{R},$$
with $h > 0$ convex and $w \geq 0$ coordinate-convex.
The main difficulty is that the equation for $w$ is not rotation-invariant.

We first solve for $h$. The solution with initial conditions $h(0) = 1, \, h'(0) = 0$ is
$$h(s) = G^{-1}(\sqrt{2}|s|),$$
where 
$$G(t) = \int_1^{t} \left(\frac{x}{x-1}\right)^{1/2}\,dx.$$
In particular, $h$ is positive, convex, globally defined and analytic, and $h \sim \sqrt{2} |s| - \frac{1}{2}\log|s|$ for $|s|$ large.

We now construct a positive coordinate-convex solution to $ww_{11}w_{22} = 1$ in $\{x_2 > 0\}$, where $w$ is homogeneous of degree $4/3$, 
even over the $x_2$ axis, and $w_1(1,\,1) = w_2(1,\,1)$. We can extend to a global solution on $\mathbb{R}^2$ by taking reflections over the lines $x_2 = \pm x_1$. 
(The solution we construct is in fact analytic outside the origin; see Remark (\ref{Gluing})). 

Let 
$$w(x_1,\,x_2) = x_2^{4/3}g(x_2^{-1}x_1)$$ 
in $\{x_2 > 0\}$, so that $w$ is $4/3$-homogeneous and $w(t,\,1) = g(t)$. 
The equation for $w$ reduces to the ODE
\begin{equation}\label{ODE}
 g\,g''\,\left(t^2g'' - \frac{2}{3}tg' + \frac{4}{9}g\right) = 1.
\end{equation}

We first claim that there exists a global even, convex solution $g_1$ to (\ref{ODE}) with the initial conditions $g_1(0) = 1$ and $g_1'(0) = 0$. The existence
and uniqueness in a neighborhood of $0$ (say $|t| < \epsilon$) follows from the fact that 
\begin{equation}\label{Poly}
 xz(t^2z - 2/3ty + 4/9x) - 1 = 0 
\end{equation}
defines $z$ as a smooth function of $(x,\,y,\,t)$
in a neighborhood of $(x,\,y,\,z,\,t) = (1,\,0,\,9/4,\,0)$ by the implicit function theorem. The solution is even by the invariance of (\ref{ODE}) under
reflection and the initial conditions, and convex since $g_1''(0) > 0$. To complete the argument, note that for $x,\,t > 0$ the positive solution to (\ref{Poly}) 
is given by
\begin{equation}\label{Soln}
 z(x,\,y,\,t) = \frac{1}{3t^2}\left(\left(ty - \frac{2}{3}x\right) + \left(\frac{9t^2}{x} + \left(ty- \frac{2}{3}x\right)^2\right)^{1/2}\right).
\end{equation}
This function is smooth and uniformly Lipschitz in $x,\,y$ in the region $\{t \geq \epsilon\} \cap \{x \geq 1\}$. We have
$g_1 \geq 1$ on any interval of existence around $0$ by the initial conditions and convexity, which combined with the previous observation gives long-time existence.

We next observe that for $\lambda > 0$ the rescalings
$$g_{\lambda}(t) := \lambda^{-2/3}g_1(\lambda t)$$
solve (\ref{ODE}). This invariance comes from the invariance of $ww_{11}w_{22} = 1$ under $(x_1,\,x_2) \rightarrow (\lambda^{1/2} x_1,\, \lambda^{-1/2}x_2)$.
We will choose $\lambda_0 > 0$ such that $g := g_{\lambda_0}$ satisfies
$$g'(1) - 2/3g(1) = 0,$$ 
which implies that $w_1(1,\,1) = w_2(1,\,1)$. 

To that end we let 
$$h_{\lambda}(t) := tg_{\lambda}'(t) - 2/3 g_{\lambda} = \lambda^{-2/3}h_1(\lambda t).$$
It suffices to show that $h_1(\lambda_0) = 0$ for some $\lambda_0 > 0$. (Then we have $h_{\lambda_0}(1) = 0$, so letting $g = g_{\lambda_0}$ would
complete the construction).

Since 
$$h_1(t) = \frac{d}{d\lambda}g_{\lambda}|_{\lambda = 1}(t)$$ 
and $g_{\lambda}$ solve (\ref{ODE}) for all $\lambda > 0$, the function $h_1$ solves the linearized equation
\begin{equation}\label{LinearizedODE}
 \left(g_1 + t^2(g_1\,g_1'')^2\right)h_1'' - \frac{2}{3}(g_1g_1'')^2th_1' + \left(g_1'' + \frac{4}{9}(g_1g_1'')^2\right)h_1 = 0.
\end{equation}
Since $h_1(0) = -2/3$ we have that $h_1$ is convex in a neighborhood of $0$. Note that $h_1$ is even. 
The equation (\ref{LinearizedODE}) prevents $h_1'$ from becoming zero before
$h_1$ reaches $0$. We conclude from (\ref{LinearizedODE}) that $h_1$ is convex in the interval around $0$ where it is negative,
and thus crosses zero at some time $\lambda_0 > 0$. This completes the construction.

\begin{rem}\label{Gluing}
 We have in addition that 
 $$g(t) = t^{4/3}g(1/t).$$ 
 Indeed, one checks that $\tilde{g}:= t^{4/3}g(1/t)$ solves (\ref{ODE}) and satisfies
 $\tilde{g}(1) = g(1),\, \tilde{g}'(1) = g'(1)$ by the condition $g'(1) - 2/3g(1) = 0$. We conclude that 
 $$g \sim a t^{4/3} + b t^{-2/3}$$ 
 for $t$ large and for some $a,\,b > 0$.
 We also conclude that 
 $$w = |x_2|^{4/3}g(x_2^{-1}x_1) = |x_1|^{4/3}g(x_1^{-1}x_2)$$ 
 is analytic outside the origin.
\end{rem}

%%%%%%%%%%%%%%%%%%%%%%%%%%%%%%%%%%%%%%%%%%%%%%%%%%%%%%%%%%%%%%%%%%%%%%%%%%%%%%%%%%
\section{Open Problems}\label{OpenProblems}

Here we list some open problems related to our results.

\begin{enumerate}
\item Classify the global solutions to $u_{11}u_{22} = 1$ in $\mathbb{R}^2$.
\item Find conditions (e.g. on the boundary and boundary data) that guarantee an interior $C^2$ estimate for $\Pi_{i = 1}^n u_{ii} = 1$
in $\mathbb{R}^n,\, n \geq 3$.
\item Solve the classical Dirichlet problem for $\Pi_{i = 1}^n u_{ii} = 1,\, n \geq 3$, on some natural class of domains. 
(More generally, consider equations for concave symmetric functions of the $u_{ii}$).
\item Analyze the structure of the singular set for solutions to $\Pi_{i = 1}^n u_{ii} = 1$ in dimensions $n \geq 3$. For example:
Are the singularities analytic? Do they propagate to the boundary? What is the Hausdorff dimension of the singular set?
\end{enumerate}

%%%%%%%%%%%%%%%%%%%%%%%%%%%%%%%%%%%%%%%%%%%%%%%%%%%%%%%%%%%%%%%%%%%%%%%%%%%%%%%%%%
\section{Appendix}\label{Appendix}

In the appendix we describe how to obtain existence in $C^{\infty}_b(\mathbb{R}^2)$ of solutions to (\ref{ApproxEquation}) using global $C^2$ estimates for 
(\ref{Model}). 

Let $w$ and $\tilde{w}$ be as in the proof of Proposition \ref{GlobalApproximateProblem}, and let
$$g = \Pi_{i = 1}^2 ((C_0\tilde{w}-K_0)_{ii} - \epsilon^{-1}(C_0\tilde{w} - K_0 - \varphi)\rho),$$
for constants $C_0,\,K_0$ chosen large enough that $g > 1$. We would like to solve in $C^{\infty}_b(\mathbb{R}^2)$ the problems
\begin{equation}\label{MOCEquation}
 \sum_{i = 1}^2 \log(u_{ii} - \epsilon^{-1}(u - \varphi)\rho) = \log(t + (1-t)g),
\end{equation}
for all $t \in [0,\,1]$.

We first claim the the set of $t$ for which (\ref{MOCEquation}) is solvable in $C^{\infty}_b(\mathbb{R}^2)$ is closed. This follows from
the global $C^2$ estimates for (\ref{Model}). Indeed, by classical uniformly elliptic theory these imply
global $C^k$ estimates (for each $k$) for solutions in $C^{\infty}_b(\mathbb{R}^2)$ of (\ref{MOCEquation}), that are independent of $t$.

We next claim that the set of $t$ for which (\ref{MOCEquation}) is solvable in $C^{\infty}_b(\mathbb{R}^2)$ is also open. 
We will use the implicit function theorem. To that end, let $X_1$ be the open subset of $C^{2,\alpha}(\mathbb{R}^2)$ given by
$$X_1 = \{\psi \in C^{2,\,\alpha}(\mathbb{R}^2): \inf_{\mathbb{R}^2} (\psi_{ii} - \epsilon^{-1}(\psi - \varphi)\rho) > 0 \text{ for } i=1,\,2\},$$
and define $G: X_1 \times [0,\,1] \rightarrow C^{\alpha}(\mathbb{R}^2)$ by
$$G(u,\,t) = \sum_{i = 1}^2 \log(u_{ii} - \epsilon^{-1}(u-\varphi)\rho) - \log(t + (1-t)g).$$
Assume that $G(u_0,\,t_0) = 0$. It is straightforward to check that $G$ is $C^1$ on $X_1 \times [0,\,1]$, 
and that the linearization $G_u$ at $(u_0,\,t_0)$ is given by
$$G_u(u_0,\,t_0)(v) = \sum_{i = 1}^2 A^{ii}v_{ii} - \epsilon^{-1}\left(\sum_{i = 1}^2 A^{ii}\right)\rho v,$$
where $A^{ii} = 1/((u_0)_{ii} - \epsilon^{-1}(u_0-\varphi)\rho)$.

The injectivity of $G_u(u_0,\,t_0) : C^{2,\,\alpha}(\mathbb{R}^2) \rightarrow C^{\alpha}(\mathbb{R}^2)$ follows from the maximum principle.
(We remark again that there is no problem with the domain being unbounded, since we work with globally bounded quantities; one can argue as in the $C^0$ 
estimate from the proof of Proposition \ref{GlobalApproximateProblem}). For surjectivity, solve the problems 
$$G_u(u_0,\,t_0)(v_R) = f, \, v_R|_{\partial B_R} = 0$$ 
for each $R$. 
The functions $\pm (-C\tilde{w} + K)$ are super- and sub- solutions for $C,\,K$ large constants depending on $\|f\|_{L^{\infty}(\mathbb{R}^2)}$, so
by the maximum principle, $v_R$ are uniformly bounded. Schauder estimates give uniform $C^{2,\,\alpha}$ bounds for $v_R$ in $B_{R-1}$.
In the limit $R \rightarrow \infty$ we obtain a solution in $C^{2,\,\alpha}(\mathbb{R}^2)$ to $G_u(u_0,\,t_0)v = f$, with $\|v\|_{C^{2,\,\alpha}(\mathbb{R}^2)}$
controlled by $\|f\|_{C^{\alpha}(\mathbb{R}^2)}$. 

By the implicit function theorem (see e.g. Chapter $17$ in \cite{GT}), there exist solutions
in $X_1$ to $G(u,\,t) = 0$ for all $t$ close to $t_0$. Schauder theory implies that in fact $u \in C^{\infty}_b(\mathbb{R}^2)$, proving the claim.

Since $G(C_0\tilde{w} - K_0,\,0) = 0$, the set of $t$ for which (\ref{MOCEquation}) is solvable in $C^{\infty}_b(\mathbb{R}^2)$
is nonempty. We conclude that the equation (\ref{MOCEquation}) is solvable in $C^{\infty}_b(\mathbb{R}^2)$ for all $t \in [0,\,1]$.

%%%%%%%%%%%%%%%%%%%%%%%%%%%%%%%%%%%%%%%%%%%%%%%%%%%%%%%%%%%%%%%%%%%%%%%%%%%%%%%%%%%
\section*{Acknowledgments}
C. Mooney was supported by NSF grant DMS-1501152 and ERC grant ``Regularity and Stability in Partial Differential Equations" (RSPDE).

O. Savin was supported by NSF grant DMS-1500438.

%%%%%%%%%%%%%%%%%%%%%%%%%%%%%%%%%%%%%%%%%%%%%%%%%%%%%%%%%%%%%%%%%%%%%%%%%%%%%%%%%%%

%%%%%%%%%%%%%%%%%%%%%%%%%%%%%%%%%%%%%%%%%%%%%%%%%%%%%%%%%%%%%%%%%%%%%%%%%%%%%%%%%%%%

\end{document}